\documentclass[12pt]{amsart}
\usepackage{amscd}
\usepackage{amsmath}
\usepackage{amssymb}
\usepackage{verbatim}
\usepackage{amsthm}
\usepackage[dvipdfmx]{hyperref}

\numberwithin{equation}{section}
\newtheorem{thm}{Theorem}[section]
\newtheorem{prop}[thm]{Proposition}
\newtheorem{lem}[thm]{Lemma}

\newtheorem{cor}[thm]{Corollary}

\newtheorem{df}[thm]{Definition}
\newtheorem{rem}[thm]{Remark}

\newcommand{\C}{\mathbb{C}}

\newcommand{\bP}{\mathbb{P}}

\newcommand{\R}{\mathbb{R}}

\newcommand{\Z}{\mathbb{Z}}

\newcommand{\cH}{\mathcal{H}}

\newcommand{\cO}{\mathcal{O}}
\newcommand{\cU}{\mathcal{U}}
\newcommand{\cV}{\mathcal{V}}

\newcommand{\divi}{\mathrm{div}}

\newcommand{\Gal}{\mathrm{Gal}}

\newcommand{\ord}{\mathrm{ord}}

\newcommand{\Spec}{\mathrm{Spec}}

\newcommand{\ovl}[1]{\overline{#1}}

\newcommand{\wt}[1]{\widetilde{#1}}

\title{Belyi's theorem in characteristic two}
\author{Yusuke Sugiyama}
\address{Yusuke Sugiyama\\
Department of Mathematics\\
Graduate School of Science\\
Osaka University\\
1-1 Machikaneyama Toyonaka Osaka 
560-0043 JAPAN}
\email{\href{mailto:u.sugiyama@gmail.com}{\texttt{u.sugiyama@gmail.com}}}

\author{Seidai Yasuda}
\address{Seidai Yasuda\\
Department of Mathematics\\
Graduate School of Science\\
Osaka University\\
1-1 Machikaneyama Toyonaka Osaka 
560-0043 JAPAN}
\email{\href{mailto:s-yasuda@math.sci.osaka-u.ac.jp}{\texttt{s-yasuda@math.sci.osaka-u.ac.jp}}}

\date{\today}

\begin{document}
\maketitle
\begin{abstract}
We prove an analogue of Belyi's theorem in characteristic two (cf.\ Section 1). Our proof consists of the following three steps. We first introduce a new notion called {\em pseudo-tame} for morphisms between curves over an algebraically closed field of characteristic two. Secondly, we prove the existence  of a ``pseudo-tame'' rational function by proving vanishing of an obstruction class. Finally we will construct a tamely ramified rational function from the ``pseudo-tame'' rational function.
\end{abstract}

\section{Introduction}

Belyi's theorem (cf.\ \cite{B}) states that, 
for a proper smooth curve $X$ over the
field of complex numbers, $X$ is defined over a number field if and only if $X$ admits a rational function $f$ on $X$ such that $f$, regarded as a morphism
from $X$ to the projective line, has at most three branch points. 
In \cite{Sa}, M.\ Sa\"idi remarked that the following analogue 
of Belyi's theorem holds in odd positive characteristics: 
for a proper smooth curve $C$ over a field of odd characteristic, $C$ is defined over a finite field if and only if $C$ admits a rational function $f$ such that $f$, regarded as a morphism from $C$ to the projective line, is tamely ramified everywhere and has at most three branch points. 
In characteristic two, it is easy to see that the ``if" part of the same statement holds true, however the ``only if" part has been remained open. In this paper, we will give a proof of the ``only if" part by proving the following
statement, which is well-known when the base field $k$ is not of
characteristic two (cf.\ \cite{F}):
\begin{thm} \label{thm:existence_tame}
Let $X$ be a proper smooth 
curve over an algebraically closed field $k$. Then $X$ 
admits a morphism $f\colon X\to \bP^{1}_{k}$ that is tamely
ramified everywhere.
\end{thm}

Let us briefly explain an organization of this paper.
In Section \ref{sec:p-tame}, we will introduce a new notion called a ``pseudo-tame" morphism and observe their basic properties.
In Section \ref{sec:beta}, we introduce a certain obstruction class to the existence of a ``pseudo-tame" rational function on a curve in characteristic two, and then prove that this obstruction class always vanishes. As a consequence, we have a ``pseudo-tame" rational function on any curve in characteristic two. In Section \ref{sec:exists}, we will construct a tamely ramified rational function from the ``pseudo-tame'' rational function.
In Section \ref{sec:upper}, we give explicit upper bounds, for any given curve $X$ in characteristic two, of the minimums of the degrees of pseudo-tame ramified rational functions on $X$ and those of tamely ramified rational functions on $X$.

\textbf{Acknowledgment}.
The authors would like to thank Akio Tamagawa for discussion 
concerning this article.
They would especially like to thank Jaap Top and Roos Westerbeek
for their careful reading of the manuscript and for
kindly having pointed out a mistake in the proof of 
Theorem \ref{thm:tame exists} in an earlier version of the
manuscript.
They would also like to thank Jaap Top for informing them
that Seher Tutdere and Nurdag\"ul Anbar have obtained 
a way, different from the one presented in this paper, 
to correct the mistake mentioned above in 
an earlier version of the manuscript.
During this research, the second author was partially supported by 
JSPS Grant-in-Aid for Scientific Research 15H03610.

\section{Pseudo-tame morphisms}\label{sec:p-tame}
From now on until the end of this paper, we fix
an algebraically closed field $k$ of characteristic two.
By a ``{\em curve}'', we mean a one-dimensional integral scheme 
which is proper and smooth over $k$. 
For a curve $X$, we denote by $k(X)$ the field of rational
functions on $X$.

\subsection{Basic facts on curves}\label{sec:basic}
In this paragraph we collect some basic facts on curves 
over an algebraically closed field of characteristic two (cf.\ \cite{H}).

Let $X$ be a curve.
%
Since the relative Frobenius on $X$ is of degree two and 
$k$ is algebraically closed (in particular perfect), 
$k(X)$ is a two-dimensional $k(X)^{2}$-vector space, where 
$k(X)^2=\{ f^2\mid f\in k(X)\}$. 
Note that the differential $dg$ vanishes on $X$ 
if and only if $g \in k(X)^{2}$.
Thus, for any $g \in k(X)$ with $dg\neq0$, $k(X)$ is the direct sum
$k(X)^{2}\oplus k(X)^{2}g$ as a $k(X)^{2}$-vector space. 

We denote by $B_X$ the sheaf $\mathcal{O}_{X}/\mathcal{O}^{2}_{X}$ 
of $\cO^2_X$-modules on $X$ and call it Raynaud's sheaf (cf.\ \cite{R}). 
The Jacobian of $X$ is ordinary if and only if $H^0(X,B_{X})=0$.

Let $\ovl{g} \in k(X)/k(X)^{2}$. We write $d\ovl{g}$ for the 
differential form $dg$, where $g \in k(X)$ is a representative 
of $\ovl{g}$ (note that $d\ovl{g}$ is independent of 
the choice of $g$).
For a closed point $x \in X$, we say that $\ovl{g}$ {\em regular at $x$} 
if the following two equivalent conditions hold:
\begin{enumerate}
\item There exists a representative $g \in k(X)$ of $\ovl{g}$ 
that is regular at $x$.
\item $d\ovl{g}\in \Omega_{X}$ is regular at $x$.
\end{enumerate}
For an open subset $U\subset X$, we identify $H^{0}(U,B_{X})$ 
with the set of elements of $k(X)/k(X)^{2}$ which are regular 
at any closed points of $U$.

\subsection{Pseudo-tame morphisms}
In this paragraph we introduce the notion ``pseudo-tame"
for morphisms of curves.

\begin{df} \label{df:p-tame_morphism}
Let $X,Y$ be curves and $f\colon X\rightarrow Y$ be a finite morphism. 
Let $x\in X$ be a closed point and set $y=f(x)$. 
Let $t\in \mathcal{O}_{Y,y}$ be an uniformizer at $y$.
We say that $f$ is {\em pseudo-tame} at $x$ if 
there exists an element $h\in \mathcal{O}_{X,x}$ such 
that $v_{x}(f^{*}t+h^{4})$ is an odd number.
One can check easily that this property is independent of the
choice of $t$.
For a nonempty open subset $U\subset X$, we say that $f$ 
is pseudo-tame on $U$ if $f$ is pseudo-tame at any closed points of $U$.
\end{df}

Note that if $f$ is tame at $x$, i.e., the morphism
$X \to \bP^1_k$ given by $f$ is at most tamely ramified at $x$
then $f$ is pseudo-tame at $x$.

\begin{df} \label{df:p-tame_function}
For a curve $X$ and $f\in k(X)$, we say that $f$ is {\em pseudo-tame} at a closed point $x\in X$ if  $f\colon X\to \bP^{1}_{k}$ is pseudo-tame at $x$. For a nonempty open subset $U\subset X$, we say that $f$ is {\em pseudo-tame on $U$} if $f$ is pseudo-tame at any closed points of $U$. 

Note that $f\in k(X)$ is pseudo-tame at $x$ if and only if $f$ turns into an uniformizer at $x$ by the linear fractional transformation of $PGL_2(k(X)^4)$ on $k(X)$ (see Remark \ref{rem:p-tame_characterization}).
\end{df}

\subsection{Properties of pseudo-tame morphisms}
Set $\cH := k(X)\setminus k(X)^{2}$.
In this paragraph, we first introduce an element 
$a(f,g) \in k(X)/k(X)^{2}$ for $f,g \in \cH$ that will play
an important role when we will define an obstruction 
class $\beta(X)\in H^{1}(X,B_{X})$. 
We then state some basic properties of pseudo-tame morphisms.

\begin{rem} \label{rem:expansion}
Let $x \in X$ be a closed point and $y=f(x)$.
Let $s,t$ be uniformizers of $\cO_{X,x}$, $\cO_{Y,y}$,
respectively. Let us consider the power series expansion 
of $f^{\ast}(t)\in \mathcal{O}_{X,x}$ with respect to $s$.
Then we have the following: 
$f$ is pseudo-tame at $x$ if and only if for any non-vanishing term 
in the power series expansion of $f^{\ast}(t)$ with degree smaller
than $ord_{x}(df^{*}t)+1$, the degree is a multiple of four.

Moreover, by considering the Laurent series expansion 
of $f\in k(X)$ at $x$ with respect to $s$, we have the following:
$f\in k(X)$ is pseudo-tame at $x$ if and only if 
for any non-vanishing term in the Laurent series expansion of $f$ with 
degree smaller than  $ord_{x}(df)+1$, the degree is a multiple of four.
\end{rem}

\begin{prop} \label{prop:p-tame_composite}
Let $X,Y,Z$ be curves and $f\colon X\rightarrow Y$, $g\colon Y\rightarrow Z$ be finite.
Let $x\in X$ be a closed point and set $y=f(x)$. 
Suppose $f,g$ are pseudo-tame at $x,y$ respectively. 
Then $g\circ f$ is pseudo-tame at $x$.
\end{prop}
\begin{proof}
The assertion follows from Remark \ref{rem:expansion} 
by looking at the formal expansions
$f$, $g$ and $g \circ f$.
\end{proof} 

We set $\Gamma :=PGL_{2}(k(X)^{4})$. 
Observe that $\mathcal{H}$ is 
the set of finite separable rational functions on $X$ and that $\Gamma$ acts on $\mathcal{H}$ freely by the linear fractional transformation. We would like to propose the following analogy.
$$
\begin{array}{c|c}
\hline
\C & k(X) \\
\hline
\C \setminus \R & \mathcal{H}=k(X)\setminus k(X)^2 \\
PGL_2(\R) & PGL_2(k(X)^2) \\
PGL_2(\Z) & \Gamma = PGL_2(k(X)^4) \\
\hline
\end{array}
$$
\begin{rem} \label{rem:p-tame_characterization}
For any $f\in \mathcal{H}$ and for any closed point $x\in X$, 
there exists $\gamma \in \Gamma$ such that $v_{x}(\gamma f)\in \{1,2\}$.
By Remark \ref{rem:expansion} above, 
we have another characterization of pseudo-tame rational functions as follows: \par An element $f\in \mathcal{H}$ is pseudo-tame at $x\in X$ if and only if there exists $\gamma \in \Gamma$ such that $v_{x}(\gamma f)=1$.
\end{rem} 
Let $f,g$ be in $\mathcal{H}$.
Since $k(X)=k(X)^{2}\oplus k(X)^{2}g$, we can find $F_{i},f_{j}\in k(X)$ such that
$$ f=F_{0}^{2}+F_{1}^{2}g=f_{0}^{4}+f_{1}^{4}g+f_{2}^{4}g^{2}+f_{3}^{4}g^{3}.$$
Then we define $a(f,g)\in k(X)/k(X)^{2}$ as follows.

\begin{df} \label{df:a(f,g)}
For $f,g\in\mathcal{H}$, we set 
\begin{align*}
A(f,g):= & ({dF_{0}}/{dg})^{2}+({dF_{1}}/{dg})^{2}g+({dF_{1}}/{dg})F_{1}\\
= & f_{1}^{2}f_{3}^{2}+f_{2}^{4}  \in k(X),
\end{align*}
and
\begin{align*}
a(f,g):= & \frac{A(f,g)g}{df/dg}\mod\, k(X)^{2}\\
=& \frac{(f_{1}^{2}f_{3}^{2}+f_{2}^{4})g}{f_{3}^{4}g^{2}+f_{1}^{4}}\mod\, k(X)^{2} \in k(X)/k(X)^{2}.
\end{align*}
\end{df}

The following property of $a(f,g)$ is important.
\begin{prop} \label{prop:1-cocycle}
For $f,g,h\in\mathcal{H}$, the following $\check{C}ech$ 
$1$-cocycle condition holds: 
$$ a(f,g)+a(g,h)+a(h,f)=0\in k(X)/k(X)^{2}.$$
\end{prop}
\begin{proof}
We note that an element $F \in k(X)$ belongs to $k(X)^{2}$ 
if and only if $dF=0\in \Omega_{X}$ and that $A(f,g),df/dg\in k(X)^{2}$. Hence it suffices to show  $$\dfrac{A(f,g)}{df/dg}dg+\dfrac{A(g,h)}{dg/dh}dh+\dfrac{A(h,f)}{dh/df}df=0\in \Omega_{X}.$$
The relation
$A(f,g)=(df/dg)^{3}A(g,f)$
reduces us to showing the equality
$$
A(g,f)+
(dh/df)^2 A(g,h)
+ (dg/dh)A(h,f)=0 \in k(X).
$$
Write $g,h$ as 
$g=G_{0}^{2}+G_{1}^{2}f$, 
$h=H_{0}^{2}+H_{1}^{2}f$, and
$g=g_{0}^{2}+g_{1}^{2}h$.
Then we have
\begin{equation} \label{eq:Agf}
A(g,f) = \left(\frac{dG_0}{df}\right)^2
+ \left(\frac{dG_1}{df}\right)^2 f
+ \frac{dG_1}{df} G_1,
\end{equation}
$$
(dh/df)^2 A(g,h) =
\left(\frac{dg_0}{df}\right)^2
+ \left(\frac{dg_1}{df} \right)^2 h
+ \frac{dg_1}{df} g_1 H_1^2,
$$
and
$$
\frac{dg}{dh} A(h,f)
=\left(\frac{dH_0}{df}\right)^2 g_1^2
+ \left(\frac{dH_1}{df}\right)^2 g_1^2 f
+ \frac{dH_1}{df} g_1^2 H_1.
$$
Hence, by applying the equalities 
$G_0 = g_0 + g_1 H_0$
and $G_1 = g_1 H_1$ to \eqref{eq:Agf},
we obtain the desired equality.
\end{proof}
Proposition \ref{prop:1-cocycle} 
provides us with some more properties of $a(f,g)$.
\begin{prop} \label{prop:a(f,g)_basic}
For $f,g\in \mathcal{H}$ and $\gamma,\delta \in \Gamma$, 
the followings are true:
\begin{enumerate}
\item $a(f,g)=a(g,f)$.
\item $a(f,g)=a(f,1/g)$.
\item For $s\in k(X)^{\times},t\in k(X)$, we have $a(f,g)=a(f,s^{4}g+t^{4})$.
\item $a(f,g)=a(\gamma f,\delta g)$.
\end{enumerate}
\end{prop}
\begin{proof}
Since $a(f,f)=0$, Proposition \ref{prop:1-cocycle} implies the equality (1).
Since $1/g= 0^{2}+(1/g)^{2}g$, we have $A(1/g,g)=0$.
Hence $a(1/g,g)=0$.
Thus by Proposition \ref{prop:1-cocycle}, we have the equality (2).
Since $s\in k(X)^{\times}$ and $g\in \mathcal{H}$, $s^{4}g+t^{4}\in \mathcal{H}$ and $a(s^{4}g+t^{4},g)=0$. Therefore we obtain the equality (3) 
from Proposition \ref{prop:1-cocycle}.
We prove (4). By Proposition 2.9.(1), it suffices to show that $a(f,g)=a(f,\gamma g)$. \\ Set $\gamma = \left[ \begin{array}{cc} a^{4}&b^{4}\\ c^{4}&d^{4} \end{array} \right] \in \Gamma =PGL_{2}(k(X)^{4})$. If $c=0$, then the assertion follows from the assertion (3). Let us assume $c\neq0$. Then we have
\begin{eqnarray*}
  a(f,\gamma g)&=&a\left(f,\frac{a^{4}g+b^{4}}{c^{4}g+d^{4}}\right) \\ 
&=&a\left(f,(a/c)^{4}+\frac{(ad-bc)^{4}}{c^{8}g+d^{8}}\right) \\
&=&a(f,(c^{8}g+d^{8})^{-1}) \\
&=&a(f,(c^2)^{4}g+(d^2)^{4}) \\ 
&=&a(f,g).
\end{eqnarray*}
This completes the proof.
\end{proof}
\begin{prop}\label{prop:a(f,g)_Gamma}
For $f,g\in \mathcal{H}$, the following two conditions are equivalent.
\begin{enumerate}
\item $a(f,g)=0 \in k(X)/k(X)^2$.
\item There exists $\gamma \in \Gamma$ such that $f=\gamma g$.
\end{enumerate}
\end{prop}
\begin{proof}
If $f = \gamma g$ for some $\gamma \in \Gamma$, then
by (4) of Proposition \ref{prop:a(f,g)_basic}
we have $a(f,g)=a(f,\gamma^{-1} f)=a(f,f)=0$.
This prove that (2) implies (1).
It remains to prove (1) implies (2). 
Note that $a(f,g):=A(f,g)g/(df/dg)\mod\, k(X)^2=0\in  k(X)/k(X)^2 $ if and only if $A(f,g)=0\in  k(X)$. 
Let us write $f=f_{0}^{4}+f_{1}^{4}g+f_{2}^{4}g^{2}+f_{3}^{4}g^{3}$. 
Then we have $f_1 f_3 + f_2^2 = 0$.
Recall that $f\in \mathcal{H}$ implies $(f_1,f_3)\neq(0,0)$.
First suppose $f_3=0$. 
Then $f_1\neq0$ and $f_2=0$. 
Thus $f=f_{0}^{4}+f_{1}^{4}g\ (f_1\neq0)$. In particular, 
$f$ belongs to the same $\Gamma$-orbit of $g$. 
Next suppose $f_3\neq0$. 
In this case, we can find $\gamma$ as
$$ 
\gamma = \left[ \begin{array}{cc} 
(f_0+f_2^3/f_3^2)^4&(f_3g+f_0f_2/f_3)^4\\ 
1&(f_2/f_3)^{4} 
\end{array} \right] \in \Gamma.
$$
Here we have $\gamma \in \Gamma$
since $\det(\gamma)=f_3^{-4}(f_3^2g+f_1^2)^4\neq0$.
\end{proof}
The next theorem relates the element $a(f,g)$ 
and the notion of pseudo-tameness.
\begin{thm} \label{thm:a(f,g)_p-tame}
For $f,g\in \mathcal{H}$ with $g$ is pseudo-tame at a closed point $x\in X$, the following two conditions are equivalent:
\begin{enumerate}
\item $a(f,g)$ is regular at $x$, or equivalently, the differential form $A(f,g)dg/(df/dg)$ associated with $a(f,g)$ is regular at $x$.
\item $f$ is pseudo-tame at $x$.
\end{enumerate}
\end{thm}
\begin{proof}
Recall that $f\in \mathcal{H}$ is pseudo-tame at $x\in X$ 
if and only if there exists $\gamma \in \Gamma$ 
satisfying $v_{x}(\gamma f)=1$,
and that $f\in \mathcal{H}$ is not pseudo-tame at $x\in X$ 
if and only if there exists $\gamma \in \Gamma$ 
satisfying $v_{x}(\gamma f)=2$.
Since $a(f,g)$ is $\Gamma$-invariant by (4) of Proposition \ref{prop:a(f,g)_basic}, we may assume that $g$ is an uniformizer at $x$.

We prove that (2) implies (1).
Since $f$ is pseudo-tame at $x$,
we can assume $f$ is an uniformizer at $x$.
Writing $f$ as $f=F_0^2+F_1^2g$,
we have $F_1\in \mathcal{O}_{X,x}^{\times}$ 
and $F_0\in {\mathfrak{m}}_{X,x}$.
Then our task is to prove that $A(f,g)dg/(df/dg)$ is regular at $x$.
This follows from 
the fact that $A(f,g)=(\frac{dF_{0}}{dg})^{2}+(\frac{dF_{1}}{dg})^{2}g+(\frac{dF_{1}}{dg})F_{1}$ and $dg\in \Omega_{X}$ are regular at $x$ 
and that $df/dg=F_1^2\in \mathcal{O}_{X,x}^{\times}$.

We prove that (1) implies (2).
Suppose $f$ is not pseudo-tame at $x$. By replacing $f$ with $\gamma f$
for some suitable $\gamma \in \Gamma$, we may assume $v_{x}(f)=2$.
Write $f$ as 
$f=F_0^2+F_1^2g=f_{0}^{4}+f_{1}^{4}g+f_{2}^{4}g^{2}+f_{3}^{4}g^{3}$. 
Then we have $F_0=f_0^2+f_2^2g\in {\mathfrak{m}}_{X,x}\setminus {\mathfrak{m}}_{X,x}^2$ and $F_1=f_1^2+f_3^2g\in {\mathfrak{m}}_{X,x}$.
This implies that $f_2\in \mathcal{O}_{X,x}^{\times}$ 
and $f_1\in {\mathfrak{m}}_{X,x}$.
Thus we have $A(f,g)=f_1^2f_3^2+f_2^4\in  \mathcal{O}_{X,x}^{\times}$ 
and $df/dg=F_1^2\in {\mathfrak{m}}_{X,x}^2$.
As a result,  $A(f,g)dg/(df/dg)$ is not regular at $x$. 
which contradicts (2).
\end{proof}

\begin{cor} \label{cor:p-tame}
For $f\in \mathcal{H}$ and a closed point $x\in X$, the following
conditions are equivalent:
\begin{enumerate}
\item $f$ is pseudo-tame at $x$.
\item There exists $\gamma \in \Gamma$ such that $\gamma f$ is pseudo-tame at $x$.
\item For any $\gamma \in \Gamma,\ \gamma f$ is pseudo-tame at $x$. 
\end{enumerate}
\end{cor}
\begin{proof}
 The implication (3)$\Rightarrow$(2) is clear. We prove (2) implies (1). 
Take some $\gamma \in \Gamma$ such that $\gamma f$ is pseudo-tame at $x$.
Proposition \ref{prop:a(f,g)_Gamma} implies $a(f,\gamma f)=0$. 
Since $a(f,\gamma f)=0$ is regular at $x$, 
Theorem \ref{thm:a(f,g)_p-tame} shows  $f$ is pseudo-tame at $x$.
It remains to prove (1) implies (3). 
Take any $\gamma \in \Gamma$. 
By Proposition \ref{prop:a(f,g)_Gamma} we have $a(\gamma f,f)=0$. 
In particular $a(\gamma f,f)=0$ is regular at $x$. 
Hence it follows from Theorem \ref{thm:a(f,g)_p-tame}
that $\gamma f$ is pseudo-tame at $x$.
\end{proof}

\section{Existence of a pseudo-tame rational function}\label{sec:beta}

\subsection{Construction of an obstruction class $\beta(X)$}
The aim of this paragraph is to introduce, for a curve $X$, a
cohomology class $\beta(X) \in H^1(X,B_{X})$.

First we remark that, for any curve $X$, one can find an
open covering $\cU=(U_{i})_{i\in I}$ of $X$ 
satisfying the following conditions:
\begin{enumerate}
\item For any $i\in I$, $U_i$ is not empty.
\item For any $i\in I$, there exists $f_i \in k(X)$
that is pseudo-tame on $U_i$.
\end{enumerate}
For example, we may take $I$ as the set of closed points of $X$, and
$\cU$ as follows: for each closed point $x\in X$, 
take an uniformizer $t_x$ at $x$ and set $U_x$:=$X\setminus{\text{ramification points of }t_x}$. We also note that, if $\cU=(U_{i})_{i\in I}$ satisfies
the two conditions above, then $U_i$ is affine for any $i \in I$.
This implies that, 
for any quasi-coherent $\cO_X^2$-module, its cohomology groups
coincide with the corresponding \v{C}ech cohomology groups 
with respect to $\cU$.

\begin{df}\label{df:beta}
For a curve $X$, take pair $(\cU,(f_i)_{i\in I})$ of an open covering
$\cU=(U_{i})_{i\in I}$ of $X$ and a family $(f_i)_{i\in I}$ of
elements of $k(X)$ with the same index as $\cU$
satisfying the conditions (1) and (2) above.
We set $\beta_{i,j}:=a(f_i,f_j)$.
Theorem \ref{thm:a(f,g)_p-tame} implies $\beta_{i,j}\in H^0(U_i \cap U_j, B_X)$.%
Thus Proposition \ref{prop:1-cocycle} tells us
that a family $(\beta_{i,j})$ defines an element of $H^1(X,B_X)$. 
We denote this element by $\beta(X,\cU,(f_i))\in H^1(X,B_X)$.
\end{df}

\begin{prop}\label{prop: well-def beta}
The element $\beta(X,\cU,(f_i))\in H^1(X,B_X)$ is independent of the choices of $\cU$ and $(f_i)$. Therefore we denote it by $\beta(X)$.
\end{prop}
\begin{proof}
First we prove that, for any fixed open covering 
$\cU=(U_i)_{i \in I}$ of $X$ satisfying the
conditions (1) and (2) above, 
the cohomology $\beta(X,\cU,(f_i))$ is independent of the choice
of $(f_i)$.
Let us take any two families $(f_i)$ and $(g_i)$ such that 
$f_i,g_i\in k(X)$ are pseudo-tame on $U_i$ for each $i\in I$.
We will check that the two families $(a(f_i,f_j))_{i,j}$ 
and $(a(g_i,g_j))_{i,j}$ 
give the same element of $H^1(\cU,B_{X})$. 
Set $H_i=a(f_i,g_i)\in k(X)/k(X)^2$ for each $i$. 
Since $f_i,g_i$ are pseudo-tame on $U_i$, 
it follows from Theorem \ref{thm:a(f,g)_p-tame} that
$H_i$ belongs to $H^0(U_i,B_{X})$.
By Proposition \ref{prop:1-cocycle} we have
$a(f_i,f_j)=H_i+a(g_i,f_j)$ and $a(g_i,g_j)=a(g_i,f_j)+H_j$.
Hence we have $a(f_i,f_j)-a(g_i,g_j)=H_i-H_j$, that implies that
the \v{C}ech $1$-cocycle $a(f_i,f_j)-a(g_i,g_j)$ is a 1-coboundary.
Therefore $(a(f_i,f_j))_{i,j}$ and $(a(g_i,g_j))_{i,j}$ 
define the same element of 
$H^1(X,B_{X})$.
Thus we may and will denote $\beta(X,\cU,(f_i))$ 
by $\beta(X,\cU)$.

It remains to show that $\beta(X,\cU)$ is independent of $\cU$.
Let $\cU=(U_i)_{i\in I}$ and $\cV=(V_j)_{j\in J}$ be two open
coverings of $X$ satisfying the conditions (1) and (2) above.
It suffices to show $\beta(X,\cU) = \beta(X,\cV)$.
First suppose that $\cV$ is a refinement of $\cU$.
Fix a map $\iota \colon J\to I$ such that $V_{j}\subset U_{\iota(j)}$. 
We choose a family $(f_i)_{i\in I}$ such that $f_i$ is pseudo-tame on $U_i$.
For each $j\in J$, we set $g_j=f_{\iota(j)}\mid_{V_{j}}$. 
Then the families $(a(f_{i_1},f_{i_2}))_{i_1,i_2\in I}$ 
and $(a(g_{j_1},g_{j_2}))_{j_1,j_2 \in J}$ define 
the elements $\beta(X,\cU)$ and $\beta(X,\cV)$, respectively.
Hence by the definition of $H^1(X,B_X)$, we have
$\beta(X,\cU) = \beta(X,\cV)$.
In general case, let us choose a common refinement $\cU'$ of $\cU$ and $\cV$.
Then we have $\beta(X,\cU)=\beta(X,\cU')=\beta(X,\cV)$, as desired.
\end{proof}

\subsection{$\beta(X)$ as an obstruction class}

Let $K$ be a field. Recall that $K$ is called quasi-algebraically closed if $K$ satisfies the following condition: if $F\in K[\bf{T}]$ is a homogeneous polynomial of $n$-variables with $\deg(f)<n$, then $F$ has a nontrivial root in $K^{\oplus n}$
Suppose moreover that $K$ is of transcendental degree one over an algebraically closed field. Then Tsen's theorem (cf. \cite[X, \S7]{Se}) states that $K$ is quasi-algebraically closed. In particular, for any curve $X$, its function field is quasi-algebraically closed.

\begin{lem}\label{lem:a(f,g)=a}
For any $f,a\in \mathcal{H}$, there exists $g\in \mathcal{H}$ such that $a(f,g)=a\in k(X)/k(X)^2$.
\end{lem}
\begin{proof}
Since $k(X)=k(X)^2\oplus k(X)^2f$, 
there exists a uniquely $b\in k(X)$ satisfying $a\equiv b^2f\mod\, k(X)^2$.
For any $g\in \mathcal{H}$, let us write $g$ as 
$g=g_{0}^{4}+g_{1}^{4}f+g_{2}^{4}f^{2}+g_{3}^{4}f^{3}$.
Recall that $a(f,g)=a(g,f)={(g_{1}^{2}g_{3}^{2}+g_{2}^{4})f}/{(g_{3}^{4}f^{2}+g_{1}^{4})}\mod\, k(X)^{2} \in k(X)/k(X)^{2}$. Thus it suffices to show that there exist $g_i\in k(X)$ such that $$(g_1,g_3)\neq(0,0) ,\ b=\frac{g_{1}g_{3}+g_{2}^{2}}{g_{3}^{2}f+g_{1}^{2}}.$$
Set $F(T_1,T_2,T_3)=T_1T_3+T_2^2+bfT_3^2+bT_1^2\in k(X)[T_1,T_2,T_3]$. Then $F$ is a homogeneous quadratic polynomial in three variables with coefficients in $k(X)$. Tsen's theorem implies $F$ has a nontrivial root in $k(X)^{\oplus 3}$. Then the root gives $(g_1,g_2,g_3)$ as desired.
\end{proof}
By the following theorem, one can regard the cohomology class $\beta(X)$ 
as an obstruction class 
for the existence of a pseudo-tame rational function.
\begin{thm}\label{thm: beta_obstruction}
Let $X$ be a curve. The followings are equivalent:
\begin{enumerate}
\item $\beta(X)=0\in H^1(X,B_X)$,
\item There exists a pseudo-tame rational function on $X$.
\end{enumerate}
\end{thm}
\begin{proof}
 First we prove (2) implies (1). 
Take a pseudo-tame rational function, $f\in k(X)$. 
Then $\beta(X)$ is defined by a pair of $(X,f)$. 
Since $a(f,f)=0$, we have $\beta(X)=0\in H^1(X,B_X)$.

 Next we prove (1) implies (2). 
Take a pair of $(\cU=(U_i)_{i\in I},(f_i)_{i\in I})$ which defines $\beta(X)$.
Since $\beta(X)=0\in H^1(X,B_X)$, for any $i\in I$ there exists $a_i\in H^0(U_i,B_X)$ satisfying $a(f_i,f_j)=a_i-a_j$.
Lemma \ref{lem:a(f,g)=a} implies that for each $i\in I$, 
there exists $g_i\in \mathcal{H}$ satisfying $a(f_i,g_i)=a_i$.
Since $a_i$ is regular on $U_i$, the function $g_i$ is pseudo-tame on $U_i$.
On the other hand, the cocycle condition shows 
\begin{eqnarray*}
a(g_i,g_j)&=&a(g_i,f_i)+a(f_i,g_j) \\
&=&a_i+(a(f_i,f_j)+a(f_j,g_j)) \\
&=&a_i+(a_i-a_j)+a_j \\
&=&0.
\end{eqnarray*}
Therefore $a(g_i,g_j)$ is clearly regular on $U_j$. That is, $g_i$ is pseudo-tame on $U_j$. Since $j$ is arbitrary, $g_i$ is pseudo-tame on $X$.
\end{proof}

\subsection{Vanishing of $\beta(X)$}\label{sec:vanishing}
In this paragraph, we prove Theorem \ref{thm:beta_vanishing},
which states that the obstruction class $\beta(X)$ vanishes for
any curve $X$.

The Serre duality gives a $k$-bilinear perfect pairing
\begin{equation}\label{eq:Serre}
(\ ,\ ) : H^0(X,B_X)\times H^1(X,B_X) \to k.
\end{equation}
In order to prove the vanishing of $\beta(X)$, we will give an explicit description of the pairing \eqref{eq:Serre}. We first give a description of $H^i(X,B_X)$ for $i=0,1$.

{\bf A description of $H^0(X,B_X)$.}
For $f\in k(X)/k(X)^2$, take a representative $\wt{f} \in k(X)$ of $f$ and write $df$ for the meromorphic differential 1-form $d\wt{f}\in \Omega_X$. Note that $df$ depends only on $f$ and is independent of the choice of $\wt{f}$. For a nonempty open subset $U\subset X$, we have 
$$H^0(U,B_X)=\{ f\in k(X)/k(X)^2 \mid df\in \Omega_X \ \mathrm{is \ regular \ on} \ U\}.$$

{\bf A description of $H^1(X,B_X)$.}
Let $X=U\cup V$ be an affine open covering. The Mayer-Vietoris exact sequence
$$H^0(U,B_X)\oplus H^0(V,B_X)\rightarrow H^0(U\cap V,B_X)\rightarrow H^1(X,B_X) \to 0$$
gives the surjection $$H^0(U\cap V,B_X)\rightarrow H^1(X,B_X)$$ which we denote by $\Psi_{B_X,U,V}$.

Then, the pairing \eqref{eq:Serre} can be explicitly described as follows.
Take any $f \in H^0(X,B_X)$ and $\alpha \in H^1(X,B_X)$.
Given an affine open covering $X=U\cup V$, we write  $\alpha=\Psi_{B_X,U,V}(g)$ for some $g \in H^0(U\cap V,B_X)$.
Let us we take representatives $\bar{f},\bar{g}\in k(X)$
of $f,g \in k(X)/k(X)^2$.
When $f=0$, we have nothing to say.
Thus we suppose that $f\neq 0$.
Then we have $g_0,g_1\in k(X)$ such that $\bar{g}=g_0^2+g_1^2\bar{f}$.
We have the following:
$$(f,\alpha)=\sum_{x\in X\setminus U}{\mathrm{Res}}_x(g_1df).$$
Here ${\mathrm{Res}}_x(g_1df)$ is the residue of $g_1df\in \Omega_x$ at a closed point $x$.

We are now ready to prove the vanishing theorem. 

\begin{thm}\label{thm:beta_vanishing}
For any curve $X$, we have $\beta(X)=0$.
\end{thm}
\begin{proof}
 Let $X$ be a curve. 
It suffices to show that $(f,\beta(X))=0\in k$ for any $f\in H^0(X,B_X)$. 
We can assume that $f\neq0\in k(X)/k(X)^2$. 
Take a representative $\wt{f}\in k(X)$ of $f$. 
Let $S$ be the set of closed points of $X$ at which $\wt{f}$ is not tame.
Then we can find $g\in k(X)$ such that $g^2+\wt{f}$ is tame on $S$. 
Let $T$ be the set of closed points of $X$ 
at which $g^2+\wt{f}$ is not tame and 
we set $U=X\setminus S$, $V=X\setminus T$. 
By the definition of $\beta(X)$, 
we have $\beta(X)=\Psi_{B_X,U,V}(a)$ for 
$a:=a(\wt{f},\, g^2+\wt{f})\in H^0(U\cap V, B_X)$. 
Recall that we have 
$$
a=({dg}/{d\wt{f}})^2\wt{f}\mod\, k(X)^2.
$$
Then we have 
$$
(f,\beta(X))=\sum_{x\in S}{\mathrm{Res}}_{x}(({dg}/{d\wt{f}})\cdot\wt{f}).
$$
Since ${\mathrm{Res}}_{x}(({dg}/{d\wt{f}})\cdot\wt{f})={\mathrm{Res}}_{x}(dg)=0$ at any closed point $x \in X$, we have $(f,\beta(X))=0$.
\end{proof}
Thus we conclude that we have a pseudo-tame rational function for any curve.
\begin{cor}\label{cor:p-tame exists}
For any curve $X$, there exists a rational function $f\in \mathcal{H}$ which is pseudo-tame on $X$.
\end{cor}

\begin{rem}
We remark that, the set of $\Gamma$-orbits of pseudo-tame rational functions on $X$ has the following structure of
$H^0(X,B_X)$-torsor: if $S$ is the $\Gamma$-orbit of a pseudo-tame rational function $f \in k(X)$ and $a \in H^0(X,B_X)$, then $a\cdot S$ is the set of $g \in H$ satisfying $a(f,g)=a$. It follows from Lemma \ref{lem:a(f,g)=a} that this set is non-empty and it follows from Proposition \ref{prop:a(f,g)_Gamma} consists of a single $\Gamma$-orbit.
The cocycle condition in Proposition \ref{prop:1-cocycle} implies that $a\cdot(b\cdot S) = (a+b)\cdot S$ for any $a,b \in H^0(X,B_X)$.

In particular, if the Jacobian of $X$ is ordinary, then the pseudo-tame rational functions on $X$ form a single $\Gamma$-orbit.
\end{rem}

\section{Existence of a tamely ramified rational function}\label{sec:exists}

In the previous Section, we proved the existence of a pseudo-tame rational function on any curve. The aim of this section is to prove, by using a ``cubing" technique, that any $\Gamma$-orbit of a pseudo-tame rational function contains a tamely ramified rational function.

We first fix and recall some notation. 
For a curve $X$, we identify the rational functions on $X$
with the morphisms from $X$ to the projective line $\bP^1_k$ over $k$. We say that a rational function on $X$ is tamely ramified if it is, as a morphism from $X$ to $\bP^1_k$, at most tamely ramified at every closed point of $X$. Under this notation, our main result can be stated as follows:

\begin{thm}\label{thm:tame exists}
Let $X$ be a curve. Let us fix a closed point $x\in X$ and set $A=\Gamma(X\setminus\left\lbrace x \right\rbrace ,\mathcal{O}_X)$. 
Let $f_0$ be a pseudo-tame rational function on $X$.
Then there exists an element $\gamma \in \Gamma = PGL_2(k(X)^4)$
such that $f := \gamma f_0 \in A$ and that $f$ is a 
tamely ramified rational function.
\end{thm}
We will give a proof of the theorem at the end of this section.
Let $g$ denote the genus of $X$.
For an element $f\in k(X)$, we denote the order of pole at $x$ by $\deg(f)$, i.e., $\deg(f):=-\ord_{x}(f)$. 

\begin{lem}\label{lem: lem1}
For a nonzero ideal $I\subset A$, we set $d:=\dim_{k}A/I$. Then for any integer $n$ with $n\geq 2g+d$, there exists an element $f \in I$ such that $\deg(f)=n$.
\end{lem}

\begin{proof}
Let $D$ be the effective divisor on Spec$(A)=X\setminus\left\lbrace x\right\rbrace $ associated with the ideal $I$. Then we have $\deg(D)=d$. The Riemann-Roch theorem implies 
$$
\dim_{k}H^0(X,\mathcal{O}(nx-D))=n-d+1-g
$$
and
$$
\dim_{k}H^0(X,\mathcal{O}((n-1)x-D))=n-d-g.
$$
In particular, we have an inequality $$\dim_{k}H^0(X,\mathcal{O}(nx-D))>\dim_{k}H^0(X,\mathcal{O}((n-1)x-D)).$$
Thus there exists $f\in k(X)^\times$ satisfying $\divi(f)+nx-D$ is effective and $\divi(f)+(n-1)x-D$ is non-effective. Then the element $f$ satisfies the above conditions.
\end{proof}

\begin{lem}\label{lem:lem2}
For an element $f \in A$, we assume that $f$ is pseudo-tame at $x$ and that the order of the differential $df$ at $x$ satisfies $-\ord_{x}(df)\geq8g$. Then there exists $h\in A$ such that $\deg(f+h^4)=-\ord_{x}(df)-1$.
\end{lem}

\begin{proof}
We denote $-\ord_{x}(df)$ by $2e$ for some integer $e$. Then we prove the Lemma by the induction for $\deg(f)-2e$. If $\deg(f)-2e<0$ then we can take $h:=0$. Otherwise, we can find an integer $d$ with $\deg(f)=4d$. By assumption, we have an inequality $d\geq 2g$. Thus Lemma \ref{lem: lem1} implies that we can find an element $h_0\in A$ with $\deg(h_0)=d$. Then for some $a\in k$, we have $\deg(f+(ah_0)^4)<4d$. By the induction hypothesis, we can take $h_1 \in A$ with $\deg(f+(ah_0)^4+h_1^4)=2e-1$, that is, $h:=ah_0+h_1$ is what we want.
\end{proof}

\begin{lem}\label{lem:lem3}
For a nonzero ideal $I\subset A$, we set $d:=\dim_{k}A/I$. 
Then for any $a\in A$, there exists $f\in A$ 
satisfying $f \equiv a \mod I$ and $\deg(f)<d+2g$.
\end{lem}

\begin{proof}
Let $f\in A$ be a representative of $(a \mod I) \in A/I$ such that $\deg(f)\leq \deg(f_1)$ for any other representative $f_1\in A $ of  $(a \mod I) $. Then the element $f\in A$ satisfies  $\deg(f)<d+2g$. In fact, suppose $\deg(f)\geq d+2g$. Then Lemma \ref{lem: lem1} gives us an element $h\in I$ with $\deg(h)=\deg(f)$. Then we can find $b\in k$ such that $\deg(f+bh)<\deg(f)$ and $f+bh\equiv a \mod I$. This contradicts to the minimality of $\deg(f)$.
\end{proof}

\begin{lem}\label{lem:lem4}
Let $f_0$ be a pseudo-tame rational function on $X$
and $r$ a positive integer with $r \ge 8g-1$.
Then there exists an element $\gamma \in \Gamma$
such that $f := \gamma f_0 \in A$ and that $f$
satisfies the following conditions:
\begin{enumerate}
\item $\deg(f)$ is odd and $\deg(f)\geq r$,
\item Any zero of $f$ is simple.
\end{enumerate}
\end{lem}

\begin{proof}
Let us write $f_0=h_0/h_1$ for some $h_0,h_1 \in A$. Then $f_1:=h_1^4f_0=h_1^3h_0 \in A$ is a pseudo-tame rational function on $X$.
Take $h_2\in A$ satisfying $-\ord_{x}(df_1)+4\deg(h_2)\geq r+1$
and set $f_2:=h_2^4f_1$. Let $2e:=-\ord_{x}(df_2)$.
%
Since we have $2e \geq r+1 \ge 8g$, 
Lemma \ref{lem:lem2} implies that we can take $h_3\in A$ 
with $\deg(f_2+h_3^4)=2e-1 \ge r$.
Then $f_3:=f_2+h_3^4\in A$ is a pseudo-tame rational function
satisfying the condition (1).
Finally by adding some constant to $f_3$, 
we obtain a rational function $f$ satisfying the desired properties.
\end{proof}

Finally we are ready to prove our main result.
%
%
\begin{proof}[{Proof of Theorem \ref{thm:tame exists}}]
Fix a pseudo-tame rational function $f\in A$ in the $\Gamma$-orbit of $f_0$
as in Lemma \ref{lem:lem4} with $r=\max(12g-2,0)$ and set $2e-1:=\deg(f)$. 
Let us denote by $Z$ the set of the zeroes of $df$ 
and for each $z\in Z$, set $2m_z:=\ord_{z}(df)$. 
Since we have $\deg(\mathrm{div}(df))=2g-2$, 
one has $\sum_{z \in Z}m_z=e+g-1$. 
Let $I\subset A$ be an ideal associated with the effective divisor 
$$
\sum_{z\in Z, \atop m_z > 1} \left( \lfloor m_z/2 \rfloor + 1 \right) z
$$
on $\Spec(A)=X\setminus\left\lbrace x \right\rbrace$. 
Then one has an inequality $\dim_k(A/I)\leq e+g-1$. 
We note that we can find an element $a\in A$ that has the
following property:
for any $h\in A$ with $h \equiv a \mod I$, $f^3+h^4$ is tame at any $z \in Z$. 

Let us fix such $a\in A$. 
Then Lemma \ref{lem:lem3} implies we have an element $h\in A$ with $h \equiv a \mod I$ and $\deg(h)<2g+(e+g-1)$. 
We note that $f^3+h^4$ is tame at any $z \in Z$. 
Now let us check that $f^3+h^4$ is tame at $x$. 
By our assumption, one has $2e-1 \geq 12g-2$. Since this implies that
$8g + 4(e+g-1) \le 3(2e-1)$, we obtain $4 \deg(h)<3\deg(f)$. Thus $f^3+h^4$ is tame at $x$. Finally we will see $f^3+h^4$ is tame at outside of $Z\cup\left\lbrace  x\right\rbrace $. The equation $d(f^3+h^4)=f^2df$ and the condition $(2)$ in Lemma \ref{lem:lem4} imply that $0\leq \ord_{y}(d(f^3+h^4))\leq 2$ for any $y \notin Z\cup\left\lbrace  x\right\rbrace $. Since $f^3+h^4$ is everywhere pseudo-tame, the inequality implies  that $f^3+h^4$ is tame at $y \notin Z\cup\left\lbrace  x\right\rbrace $. Thus we conclude that $f^3+h^4$ is a tamely ramified rational function on $X$. 
\end{proof}

\begin{proof}[{Proof of Theorem \ref{thm:existence_tame}}]
The assertion follows from \cite{F} if the characteristic of $k$
is not equal to $2$. When $k$ is of characteristic two,
the assertion follows from 
Corollary \ref{cor:p-tame exists} and  Theorem \ref{thm:tame exists}.
\end{proof}

\section{An upper bound of the minimum degree}\label{sec:upper}

\begin{thm}\label{thm:upper}
Let $X$ be a curve of genus $g \ge 1$, and let $x \in X$
be a closed point. Let $A = \Gamma(X \setminus \{x\}, \cO_X)$.
Then there exists an element $f \in A$ such that
$\deg(f) \le 48 g^2 + 22g -1$ and that $f$ is a pseudo-tame
function on $X$.
\end{thm}

\begin{proof}
By Lemma \ref{lem: lem1}, one can choose two elements $f,t \in A$
with $\deg(f)=2g+1$ and $\deg(t)=2g$.
By adding to $t$ a suitable element of smaller degree if necessary,
we may and will assume that $k(X)$ is a separable extension of $k(t)$.
%
%
Let $Z$ be the set of zeros of $df$.
For $z \in Z$, set $2m_z := \ord_z(df)$.
Let $I$ denote the ideal of $A$ associated with
the effective divisor $\sum_{z \in Z}(m_z+1)z$.
Then we can find an element $a\in A$ that has the
following property:
for any $h\in A$ with $h \equiv a \mod I$, 
$f+h^2$ is tame at any $z \in Z$. 
The argument in the proof of Theorem \ref{thm:tame exists} shows that
$\dim_k(A/I) \le 4g$.
Hence it follows from Lemma \ref{lem:lem3} one can find
$h \in A$ with $\deg(h)<6g$ such that
$f' := f+h^2$ is tame on $U := X \setminus \{x\}$.
Set $V=X \setminus Z$. It follows from the definition
that $\beta(X)$ is equal to the image of $a(f,f') \in H^0(U\cap V,B_X)$
under the map $\Psi_{B_X,U,V}$ in Section \ref{sec:vanishing}.
Since $\beta(X)=0$, there exists an element 
$a \in H^0(U,B_X) = A/A^2$ such that $a + a(f,f')$ is regular
at $x$.
Choose a unique $b \in k(X)$ satisfying
$a = b^2 f \mod{k(X)^2}$ and let us consider the
polynomial $F(T_1,T_2,T_3) = T_1T_3 + T_2^2 +bf T_3^2
+ b T_1^2$.
Then the argument in the proof of Lemma \ref{lem:a(f,g)=a} 
and Theorem \ref{thm: beta_obstruction} 
shows that, for any non-trivial solution $(g_1,g_2,g_3) \in A^3$ of 
$F(T_1,T_2,T_3)=0$, the element
$g_1^4 f + g_2^4 f^2 + g_3^4 f^3 \in A$ is a pseudo-tame
rational function on $X$.

Observe that $a(f,f') = (dh/df)^2 f \mod{k(X)^2}$
and that $\deg(dh/df) \le 4g-1$.
The latter implies $\deg((dh/df)^2 f) \le 10g -1$.
Since $10g-1 \ge 4g$, it follow from Lemma \ref{lem: lem1}
that we can find a lift $\wt{a}\in A$ of $a$
satisfying $\deg(\wt{a}) \le 10g-1$.
Observe that $b^2 = d \wt{a}/df$.
This implies $\deg(b) \le 4g-1$.
Since $\sum_{z \in Z} m_z = 2g$, it follows from
Lemma \ref{lem: lem1} that there exists a non-zero $c \in A$
with $\deg(c) \le 4g$ satisfying $b c \in A$.
Then we have $\deg(bc) \le 8g-1$.
Hence $cF(T_1,T_2,T_3)$ is a homogeneous quadratic polynomial
with coefficients in $A$ and the degree of each coefficient
is at most $10g$.

Note that, for any non-zero $\phi \in k[t]$ and 
$i \in \{0,1,\ldots,2g-1\}$, the degree of $\phi f^i$
is congruent to $i$ modulo $2g$.
This implies that $1,f,\ldots,f^{2g-1}$ are
linearly independent over $k[t]$.
Let $L$ be the Galois closure of $k(X)$ over $k(t)$
in a separable closure of $k(X)$, and let us consider 
the polynomial
$$
G = \prod_\sigma
\sigma(F)\left(\sum_{i=0}^{2g-1} \sigma(f)^i S_{1,i}, 
\sum_{i=0}^{2g-1} \sigma(f)^i S_{2,i},
\sum_{i=0}^{2g-1} \sigma(f)^i S_{3,i}
\right),
$$
where $\sigma$ runs over the element 
of $\Gal(L/k(t))/\Gal(L/k(X))$, in variables
$S_{1,i}$, $S_{2,i}$, $S_{3,i}$ ($i \in \{0,1,\ldots, 2g-1\}$).
Then $G$ is a homogeneous polynomial
of degree $4g$ in $6g$ variables,
and the coefficients of $G$ are polynomials in $k[t]$
whose degrees are at most 
$2(2g-1)\deg f + 10g = 8g^2 + 10g -2$.
Hence the argument in the proof of Tsen's theorem
(cf.\ \cite[Chap.\ 1, 6.2, Cor.\ 1.11]{Sh}) 
shows that there exists a
non-trivial solution of $G=0$ in
$k[t]^{6g}$ whose all entries are polynomials in $t$
of degree at most $4g+2$.
Since $1,f,\ldots,f^{2g-1}$ is
linearly independent over $k[t]$, this shows that
there exists a non-trivial solution
$(g_1,g_2,g_3) \in A^3$ of $F(T_1,T_2,T_3)=0$
satisfying $\deg(g_i) \le (2g-1)\deg(f) + (4g+2) \deg(t)
= 12g^2 + 4g -1$.
Thus one can find a pseudo-tame rational function
$g' = g_1^4 f + g_2^4 f^2 + g_3^4 f^3 \in A$ on $X$
of degree at most $4(12g^2+4g-1)+3\deg(f) =
48g^2 + 22g -1$, as desired.
\end{proof}

\begin{cor}
Let $X$, $x$, and $A$ be as in Theorem \ref{thm:upper}.
Then there exists an element $f \in A$ such that
$\deg(f) \le 144 g^2 + 66g -3$ and that $f$ is a 
tamely ramified rational function on $X$.
\end{cor}

\begin{proof}
By Theorem \ref{thm:upper}, one can take 
$f \in A$ such that
$\deg(f) \le 48 g^2 + 22g -1$ and that $f$ is a pseudo-tame
function on $X$.
By replacing $f$ with $f h^4$ for some suitable $h \in A$
with $\deg(h) = 3g$ if $-\ord_x(df) < 8g$, and then by
applying Lemma \ref{lem:lem2}, we obtain
an element $f' \in A$ such that $\deg(f')$ is odd
with $12g < \deg(f') \le 48 g^2 + 22g -1$ and that 
$f'$ is a pseudo-tame function on $X$.
Then the argument of the proof of Lemma \ref{lem:lem4} 
and Theorem \ref{thm:tame exists}
shows that a function of the form 
$(f'+c)^3 + h'^4$ for some $c \in k$ and
$h' \in A$ with $4 \deg(h') < 3 \deg(f')$ is a tamely
ramified rational function on $X$, which proves the
assertion.
\end{proof}

\end{document}